\documentclass[12pt]{amsart} 
\usepackage{textcomp}

\usepackage[T1]{fontenc}
\usepackage{palatino}        
\usepackage{mathtools}
\usepackage[euler-digits]{eulervm}
\usepackage{microtype}
\usepackage{tikz-cd}      

\usepackage{setspace}
\usepackage{array, amsmath, amscd, amssymb, amsthm}
\usepackage{hyperref}

\let\oldTitle\title
\renewcommand{\title}[1]{\newcommand{\myTitle}{#1}\oldTitle{#1}} 
\title{Logarithmic Cobordism and Donaldson-Thomas Invariants}

\input xy
\xyoption{all}

\numberwithin{equation}{section}

\theoremstyle{plain}
\newtheorem{theo}{Theorem}[section]
\newtheorem{prop}[theo]{Proposition}

\theoremstyle{definition}
\newtheorem{defn}[theo]{Definition}
\newtheorem{conj}[theo]{Conjecture}
\theoremstyle{remark}
\newtheorem{rem}[theo]{Remark}
\newtheorem{eg}[theo]{Example}

\begin{document}

\author{Jose Luis Guzman}

\address{Department of Mathematics, MIT,
Cambridge, MA 02138, USA} \email{joselg@mit.edu}
\vspace{-1em}

\begin{center}
  {\small \textit{Preliminary version}}
\end{center}
\vspace{1em}

\begin{abstract}   

Let $X$ be a smooth projective 3-fold and let $D$ be a simple normal crossings (snc) divisor in $X$. In recent work by Maulik-Ranganathan, logarithmic Donaldson-Thomas invariants for the pair $(X,D)$ are defined. One particular output of this construction is a logarithmic Hilbert scheme of points $\text{Hilb}^n(X,D)$, which is a compatification of the non-compact $\text{Hilb}^n(X\setminus D)$ and points in the boundary of this compactification correspond to subschemes in \textit{expansions} $\mathcal{X} \rightarrow X $ of X. The logarithmic Hilbert scheme of points carries a zero dimensional virtual fundamental class, so we can form the generating function of the degrees of the virtual classes: $Z_{\text{DT}}(X,D;q)_0 = 1 + \sum_{n\geq1} \deg [\text{Hilb}^n(X,D)]^{\text{vir}} q^n$. We prove paper that this generating function is given by the simple formula $Z_{\text{DT}}(X,D;q)_0 = M(-q)^{\int_X c_3(T^{\text{log}}_X\otimes K_X^{\text{log}}) }$, where $M(q) = \prod_{n\geq1} \frac{1}{(1-q^n)^n}$, which was conjectured in \cite{MR}. Following the cobordism strategy of Levine-Pandharipande \cite{LP} for the conjecture when $D$ is empty, we define a logarithmic cobordism ring and prove that in dimension 3 the ring is generated by pairs $(X,D)$ on which the conjecture is known to be true, which together with the degeneration formalism completes the proof of Maulik-Ranganathan's conjecture. Many invariants in logarithmic geometry are invariant under \textit{logarithmic modifications} $(X',D')\rightarrow (X,D)$, which are blow ups of $(X,D)$ along strata. Thus, from the perspective of logarithmic geometry a cobordism ring incorporating the relation $(X',D') = (X,D)$ is more natural. We prove in Theorem 2.6 that imposing this relation collapses the theory to the empty boundary cobordism theory of  Levine-Pandharipande \cite{LP}.
\end{abstract}
\maketitle

\section{Introduction}
\subsection{Overview}
Let $X$ be a smooth projective 3-fold. The Donaldson-Thomas invariants of $X$ are obtained by performing intersection theory on the DT moduli space $\text{DT}_{\beta,\chi}(X)$, which is the Hilbert scheme of curves in $X$ with curve class $\beta\in H_2(X,\mathbb{Z})$ and holomorphic euler characteristic $\chi$. The construction of the virtual fundamental class on the DT moduli spaces also furnishes a zero dimensional virtual fundamental class for $\text{Hilb}^n(X)$. We can then form the generating series
\begin{align*}
    \text{Z}(X)= 1 + \sum_{n\geq 1} \deg [\text{Hilb}^n(X)]^{\text{vir}} q^n
\end{align*}
In \cite{MNOP01} the following was conjectured, and the conjecture was proven by Levine-Pandharipande in \cite{LP} and by Li \cite{Li06}:
\begin{theo}[\textup Levine and Pandharipande \cite{LP}\textup, \textup Li \cite{Li06}\textup]
Let $X$ be a smooth projective threefold, then we have an equality of generating functions 
\begin{equation*}
    \text{Z}(X) = M(q)^{\int_X c_3(T_X\otimes K_X)}
\end{equation*}
where $M(q)$ is the MacMahon function
\begin{equation*}
    M(q) = \prod_n \frac{1}{(1-q^n)^n}
\end{equation*}
\end{theo}   
A key technique in the study of Donaldson-Thomas theory is the use degenerations, which was introduced by Li-Wu in \cite{Wu}. The set up is as follows: we're given a smooth projective 4-fold $\mathcal{X}$ together with a flat morphism $\pi: \mathcal{X}\rightarrow \mathbb{P}^1$ where a general fiber is a smooth threefold $X$ and the central fiber $\pi^{-1}(0) = A\bigcup_D B$ is the union of two smooth divisors $A,B$ in $\mathcal{X}$ that meet transversally along a smooth divisor $A\cap B = D$. The degeneration $\pi$ also induces a degeneration of DT moduli spaces, and in particular we get a degeneration 
\begin{align*}
    \text{Hilb}^n(X) \rightsquigarrow \bigcup_{a+b = n}\text{Hilb}^a(A,D)\times \text{Hilb}^b(B,D)
\end{align*}
where $\text{Hilb}^a(A,D)$ is the Hilbert scheme of $n$ points of $A$ \textit{relative} to the smooth divisor $D$.  The relative Hilbert schemes also carry a zero dimensional virtual fundamental class, and for a pair of 3-fold and smooth divisor $(Y,E)$ we can form the relative generating function
\begin{align*}
    \text{Z}(Y,E) = 1 +\sum_{n\geq1} \deg [\text{Hilb}^n(Y,E)]^{\text{vir}} q^n
\end{align*}
The degeneration formula of \cite{Wu} applied to our degeneration $\pi:\mathcal{X}\rightarrow \mathbb{P}^1$ tells us that 
\begin{equation*}
    \text{Z}(X) = \text{Z}(A,D)\cdot \text{Z}(B,D)
\end{equation*}
With the introduction of the the technique of degenerations in DT theory, the following conjecture for the relative zero dimensional DT generating function was made in \cite{MNOP02} and was proven by Levine-Pandharipande:
\begin{theo}[\textup Levine and Pandharipande \cite{LP}\textup]
Let $D\subset X$ be a smooth divisor in a smooth projective threefold, then we have an equality of generating functions 
\begin{equation*}
    \text{Z}(X,D) = M(q)^{\int_X c_3(T^\text{log}_X\otimes K^\text{log}_X)}
\end{equation*}
where $M(q)$ is the MacMahon function, $T^\text{log}_X$ and $K^\text{log}_X$ are the logarithmic tangent bundle and logarithmic canonical bundle with respect to the divisorial log structure given by $D\subset X$
\end{theo}   
The degeneration formalism opens the following path towards a proof of Theorem 1.1. The idea is that if one could prove the conjecture for simple 3-folds, for example toric 3-folds, and then show that any 3-fold can be degenerated to toric pieces, then the degeneration formalism would allow one to conclude Theorem 1.1. This is the strategy followed by Levine-Pandharipande in \cite{LP}.

Having introduced the general background on zero dimensional DT invariants, we now explain the logarithmic DT theory portion of this paper. In \cite{MR} Maulik-Ranganathan construct logarithmic DT spaces for a pair $(X,D)$ of a smooth threefold and an snc divisor $D\subset X$. One particular output of the theory is a logarithmc hilbert scheme of points $\text{Hilb}^n(X,D)$ together with a zero dimensional virtual fundamental class, and we can again form the generating function 
\begin{align*}
    \text{Z}(X,D) = \sum_n \deg [\text{Hilb}^n(X,D)]^{\text{vir}} q^n
\end{align*}
For a simple normal crossings pair $(X,D)$, the following conjecture was made in \cite{MR} (note that it agrees with Theorem 1.2 in the case when $D$ is smooth):
\begin{conj}[\textup Maulik and Ranganathan \cite{MR}\textup]
Let $D$ be a simple normal crossings divisor in a smooth projective threefold $X$, then we have an equality of generating functions
\begin{equation*}
    Z(X,D) = M(q)^{\int_X c_3(T^\text{log}_X\otimes K^\text{log}_X)}
\end{equation*}
\end{conj}
The zero dimensional log DT generating function also satisfies a degeneration formula proven in \cite{MR23}, which we now explain. Let $\mathcal{X}\xrightarrow{\pi} \mathbb{P}^1 $ be a flat morphism from a smooth projective fourfold. Furthermore, let $\mathcal{D}$ be an snc divisor in $\mathcal{X}$ such that $\pi|_{\mathcal{D}}$ is flat and such that $\mathcal{D}\cap \pi^{-1}(t)$ is an snc divisor in $\pi^{-1}(t)$ for all $t$. Also assume that $\pi^{-1}(0) = A\bigcup_E B$ is a singular fiber which is the union of two smooth threefolds in $\mathcal{X}$ that meet transversally along a divisor $E$, with $\mathcal{D}$ meeting $E$ transversally. Let $X$ be a smooth fiber of $\pi$, which we can equip with the snc divisor $D = X\cap \mathcal{D}$, and we equip $A$ with the snc divisor $(\mathcal{D}\cap A) \cup E = D_A\cup E$ and equip $B$ with the snc divisor $(\mathcal{D}\cap B)\cup E = D_B\cup E$. Then the degeneration formalism of log DT invariants tells us that:
\begin{equation*}
    \text{Z}(X,D) = \text{Z}(A,D_A\cup E)\cdot \text{Z}(B,D_B\cup E)
\end{equation*}
\\
Our approach to Conjecture $1.3$ is inspired by the approach of Levine-Pandharipande of defining a cobordism ring of varieties and quotienting by relations obtained via double point degenerations. In particular, we define a logarithmic cobordism ring $\omega^{\text{log}}$ generated by pairs $(X,D)$ of a smooth varietie with an snc divisor $D\subset X$. The logarithmic cobordism ring is graded $\omega^{\text{Log}} = \oplus_n \omega^{\text{Log}}_n$ by the dimension of $n = \dim X$ and we prove that $\omega^{\text{log}}_3$ has a generating set given by $(\mathbb{P}^3, \emptyset), (\mathbb{P}^2\times\mathbb{P}^1,\emptyset), (\mathbb{P}^1\times\mathbb{P}^1\times\mathbb{P}^1,\emptyset), (\mathbb{P}^3,\mathbb{P}^2), (\mathbb{P}^2\times\mathbb{P}^1,\mathbb{P}^1\times\mathbb{P}^1)$; where $\emptyset$ means that our 3-fold has no boundary.

In particular, Conjecture $1.3$ is true for each generating element, and since the quantity $\int_X c_3(T^\text{log}_X\otimes K^\text{log}_X)$ is preserved under double point degenerations of snc pairs, this completes the proof of Conjecture 1.3.

\subsection{Acknowledgments} I am very grateful to my advisor Davesh Maulik, whose
guidance and support is invaluable to me.

\section{The logarithmic algebraic cobordism ring $\omega^{\text{log}}$}

We now define a cobordism ring, similar to that of Levine-Pandharipande. In our case, for each pair of smooth variety $X$ and an snc divisor $D\subset X$ we will get an element $(X,D)\in \omega^{\text{log}}$. We also allow our varieties to have no boundary, so that we have elements $(X,\emptyset)$. To define this ring, we first define how to get relations between pairs via degenerations.

\subsection{Logarithmic cobordism ring}
\begin{defn}Let $\pi:\mathcal{X}\rightarrow \mathbb{P}^1$ be a flat family of $n$-dimensional varieties such that
\begin{enumerate}
	\item $\pi^{-1}(0)= A\cup_{E}B$  is the union of two smooth varieties who intersect transversally along a smooth divisor $E = A\cap B$
	\item $\mathcal{X}$ is smooth and the morphism $\pi$ is smooth away from $0\in\mathbb{P}^1$ 
\end{enumerate}
Furthermore, suppose we are given a a family of simple normal crossings divisors, i.e. a simple normal crossings divisor $\mathcal{D}\subset \mathcal{X}$ such that
\begin{enumerate}
    \item The morphism $\mathcal{D}\xrightarrow{\pi}\mathbb{P}^1$ is flat 
    \item The intersection $\mathcal{D}\cap \pi^{-1}(t)$ is an snc divisor in $\pi^{-1}(t)$ for all $t\in \mathbb{P}^1$. In particular, we have snc divisors $D = X\cap \mathcal{D}$ and $D_1= A\cap\mathcal{D}$ and $D_2 = B\cap\mathcal{D}$. 
    \item $\mathcal{D}$ meets the double point locus $A\cap B = E$ transversally
\end{enumerate}
The \textit{double point relation} defined by such data is
\begin{equation*}
    [X,D] = [X_1, D_1 + E] +[X_2, D_2 +E]
\end{equation*}
Furthermore, if $\mathcal{X}$ has a line bundle $\mathcal{L}$, we can also extract the relation 
\begin{equation*}
    [X,D,\mathcal{L}|_X] = [X_1, D_1 + E, \mathcal{L}|_{X_1}] +[X_2, D_2 +E,\mathcal{L}|_{X_2}]
\end{equation*}
 
\end{defn}

Let $\mathcal{M}$ be the $\mathbb{Q}$-vector space spanned by pairs $[X,D]$ of smooth projective varieties and simple normal crossings divisors $D\subset X$, and 
let $\mathcal{R}$ be the subgroup spanned by all double point relations. 
We define the \textit{logarithmic cobordism group} to be $$\omega^{\text{Log}}=\mathcal{M}/\mathcal{R}.$$ 
For an snc pair $D\subset X$ the corresponding element of $\omega^{\text{Log}}$ is denoted $(X,D)\in \omega^{\text{Log}}$. Note that the $\mathbb{Q}$-vector space $\omega^{\text{Log}}$ has a ring structure given by $$(X,D)\cdot (Y,E) = (X\times Y, (D\times Y)\cup (X\times E))$$ 

Furthermore, we also have a logarithmic cobordism ring of snc pairs with $r$ line bundles $(X,D,L_1,...,L_r)\in \omega^{\text{Log}}_{1^r}$ where now relations are again obtained by degenerations $\mathcal{X}\rightarrow \mathbb{P}^1$ but now the total space of the degeneration must carry line bundles $\mathcal{L}_1,...,\mathcal{L}_r\in \text{Pic}(\mathcal{X})$. The ring $\omega^{\text{Log}}$ is graded by dimension of the variety $X$, and for an element $(X,D)$ and we let $\omega^{\text{Log}}_n$ be the corresponding piece. Some simple relations in the logarithmic cobordism ring that we will use later on are the following
\begin{eg}
We have $(\mathbb{P}^1,0) = \frac{1}{2}(\mathbb{P}^1,\emptyset)$. This is proven by taking degeneration to the normal cone for the divisor $0\in\mathbb{P}^1$ to conclude $(\mathbb{P}^1,\emptyset) =(\mathbb{P}^1,0) + (\mathbb{P}^1,0) $
\end{eg}
\begin{eg}
Let $D\subset X$ be a smooth divisor, then degeneration to the normal cone of $D$ gives us $(X,\emptyset) = (X,D) + (\mathbb{P}(N_{D\subset X}\oplus \mathcal{O}), \mathbb{P}(N_{D\subset X}))$ where $\mathbb{P}(N_{D\subset X}) = D$. In a later section, we will explain how the term $(\mathbb{P}(N_{D\subset X}\oplus \mathcal{O}),D)$ can be thought of as $(D,L)$, i.e. $D$ together with the line bundle $L = N_{D\subset X}$.
\end{eg}

The degeneration formalism of logarithmic DT invariants proven in \cite{MR23} give us the following:

\begin{prop}
There is a ring homomorphism $F:\omega^{\text{Log}}_3\rightarrow \mathbb{Q}[[x]]$ sending $(X,D)$ to the zero dimensional DT generating function $\text{Z}(X,D)$\end{prop}      

From the perspective of logarithmic geometry, the more natural ring to consider is the quotient of $\omega^{\text{Log+Mod}}$ in which we impose the relation that $(X',D') = (X,D)$ for any logarithmic modificatio $(X',D')\rightarrow (X,D)$. A logarithmic modification of $(X,D)$ is a any snc pair obtained from $(X,D)$ by blowing up along strata. Many invariants in logarithmic geometry are invariant under log modifications, and a canonical example of this is the following theorem for logarithmic Gromov-Witten invariants:
\begin{theo}[\textup Abramovich-Wise \cite{AW18}\textup]
Let $h:Y\rightarrow X$ be a logarithmic modification of logarithmically smooth schemes inducing a projection $\pi: \overline{\mathcal{M}}(Y)\rightarrow \overline{\mathcal{M}}(X)$. Then 
\begin{equation*}
    \pi_\ast\left( [\overline{\mathcal{M}}(Y)]^{\text{vir}} \right) = [\overline{\mathcal{M}}(X)]^\text{vir}
\end{equation*}
\end{theo}  

Thus we can define a logarithmic cobordism ring $\omega^{\text{log + mod}}$ that imposes the relation $(X',D') = (X,D)$ where $(X',D')$ is a logarithmic modification of $(X,D)$. In logarithmic geometry one expects that (logarithmic) enumerative invariants of a pair $(X,D)$ should be recoverable from the invariants of $X$ and its strata (each with no boundary), and the following result is part of the same general philosophy:
\begin{theo}
Consider the diagram 
\begin{center}
\begin{tikzcd}
\omega^{LP} \arrow[d] \arrow[rd] &                    \\
\omega^{Log} \arrow[r]           & \omega^{Log + Mod}
\end{tikzcd}
\end{center}
where the map $\omega^{LP}$ is the morphism from the Levine-Pandharipande cobordism ring to the logarithmic cobordism ring, and the map $\omega^{Log}\rightarrow \omega^{Log +Mod}$ is the quotient map. Then the composition $\omega^{LP}\rightarrow \omega^{Log+Mod}$ is an isomorphism.
\end{theo}

\section{Generators for $\omega^{\text{Log}}_3$}

We now outline our strategy for proving the zero dimension log DT conjecture. Our proof makes use of the fact that the logarithmic cobordism ring respects log chern numbers. What we mean is that if $(X,D)\in \omega^{\text{Log}}$ is an snc pair of dimension $n$ and we let $T_{(X,D)}$ denote the logarithmic tangent bundle, then for any partition $\lambda$ of $n$ we have a ring homomorphism 
\begin{equation*}
    c_\lambda: \omega^{\text{Log}}_n\rightarrow\mathbb{Q},\quad (X,D)\mapsto c_\lambda(T_{(X,D)}) = \prod_i c_{\lambda_i}(T_{(X,D)})
\end{equation*}
where the $\lambda_i$ are the parts of $\lambda$. In particular, if a 3-fold pair $(X,D)$ can be expressed as a rational sum 
\begin{equation*}
    (X,D) = \sum_i r_i(Y_i,E_i)
\end{equation*}
such that the zero dimensional log DT conjecture holds for each $(Y_i,E_i)$, then the conjecture must also be true for $(X,D)$. Indeed, by clearing denominators we can find integers $r,a_i$ such that 
\begin{equation*}
    a(X,D) = \sum_i a_i(Y_i,E_i)
\end{equation*}
Which then tells us
\begin{align*}
    Z(X,D)^a &= \prod_i M(q)^{a_i \int_{Y_i}c_3(T^\text{log}_{Y_i}\otimes K^{\text{log}}_{Y_i})}  \\
    &=  M(q)^{\sum_i a_i \int_{Y_i}c_3(T^\text{log}_{Y_i}\otimes K^{\text{log}}_{Y_i})} 
\end{align*}
The quantity $\int_X c_3(T^\text{log}_X\otimes K^{\text{log}}_X)$ is an invariant $\omega^{\text{Log}}_3\rightarrow \mathbb{Q}$ of $\omega^{\text{Log}}_3$ because it can be expressed as an integer sum of the basic log chern invariants $c_\lambda(T^{\text{log}})$, which tells us that 
\begin{equation*}
c_3(T^{\text{log}}_X\otimes K^{\text{log}}_X) = \sum_i r_i c_3(T^{\text{log}}_Y\otimes K^{\text{log}}_Y)
\end{equation*}
where all log structures are taken with respect to the corresponding boundary. So we get 
\begin{equation*}
    Z(X,D)^a = M(-q)^{a \int_X c_3(T^\text{log}\otimes K^{\text{log}}_X)}
\end{equation*}
Since $Z(X,D)(0) =1$ and $M(0) = 1$, then we conclude 
\begin{equation*}
    Z(X,D) = M(-q)^{ \int_X c_3(T^\text{log}\otimes K^{\text{log}}_X)}
\end{equation*}
Therefore, our strategy will be the prove that $\omega^{\text{Log}}_3$ is generated by pairs on which the conjecture is known to be true. In particular, by the previous work outlined in the introduction, the conjecture is known to be true for pairs without boundary $(X,\emptyset)$ and for \textit{smooth pairs} $(X,D)$, by which we mean that $D\subset X$ is a smooth divisor. With this in mind, to prove the conjecture it suffices to prove the following

\begin{theo}
The dimension three piece $\omega^{\text{Log}}_3$ of the logarithmic cobordism ring is generated over $\mathbb{Q}$ by the pairs $(\mathbb{P}^3,\emptyset)$, $(\mathbb{P}^2\times\mathbb{P}^1, \emptyset)$, $(\mathbb{P}^1\times\mathbb{P}^1\times\mathbb{P}^1,\emptyset)$, $(\mathbb{P}^3,\mathbb{P}^2)$ and $(\mathbb{P}^2\times\mathbb{P}^1, \mathbb{P}^1\times\mathbb{P}^1)$.
\end{theo}  

Our proof of Theorem 3.1 will be via an inductive proof on the number of irreducible components $D=D_1\cup\dots \cup D_r$ of the boundary divisor and we will also make use of results of Lee-Pandharipande \cite{LeeP} on algebraic cobordism of varieties with bundles. 
\\
The basic idea of our strategy is to use degenerations to trade boundary divisor components for line bundles, and we now outline this. We know that Conjecture 1.3 is true for smooth pairs and pairs with no boundary, so suppose we have a threefold with an snc divisor $D'= D\cup E$ that has two irreducible components. Consider the deformation to the normal cone for $E\subset X$
\begin{equation*}
    \mathcal{X} = \text{Bl}_{E\times 0} X\times \mathbb{P}^1\xrightarrow{\pi} \mathbb{P}^1
\end{equation*}
The central fiber is $\pi^{-1}(0) = X\cup_E \mathbb{P}(L\oplus \mathcal{O})$ where $L$ is the normal bundle of $E$ inside of $X$. We let $\eta: \mathbb{P}(L\oplus \mathcal{O})\rightarrow E$ denote the bundle map. For the family of snc divisors, consider the strict transform $\mathcal{D} = \widetilde{D\times\mathbb{P}^1}$. Let $Z = D\cap E$, then the relation we extract from the degeneration $\mathcal{D}\subset \mathcal{X}\rightarrow \mathbb{P}^1$ is 
\begin{equation*}
    (X,D) = (X,D\cup E) + (\mathbb{P}(L\oplus \mathcal{O}), E\cup \eta^{-1}(Z))
\end{equation*}
where $E$ sits inside of $\mathbb{P}(L\oplus\mathcal{O})$ as $\mathbb{P}(L)$. The above equality, together with the fact that the conjecture is known on $(X,D)$, tells us that to prove Conjecture 1.3 for $(X,D+E)$ it suffies to prove the conjecture for $(\mathbb{P}(L\oplus \mathcal{O}), E\cup \eta^{-1}(Z))$. The idea now is that to degenerate $(\mathbb{P}(L\oplus \mathcal{O}), E\cup \eta^{-1}(Z))$ it suffices to degenerate the triple $(E,Z,L)$, so we pass to ring $\omega^{\text{Log}}_{1^1}$ of varieties with boundary divisor and a single line bundle. Formally speaking, the reason for why it suffices to degenerate the tripple $(E,Z,L)$, rather than $(\mathbb{P}(L\oplus \mathcal{O}), E\cup \eta^{-1}(Z))$, is because there is a morphism 
\begin{equation*}
    \omega^{\text{Log}}_{1^r}\rightarrow \omega^{\text{Log}}_{1^{r-1}},\quad (X,D,L_1,...,L_r)\mapsto (\mathbb{P}(L_r\oplus \mathcal{O}), X\cup \eta^{-1}(D), \eta^\ast L_1,...,\eta^\ast L_{r-1})
\end{equation*}
This procedure of trading irreducible components for line bundles allows us to make use of results of Lee-Pandharipande on the algebraic cobordism of varieties with line bundles.

We now begin to prove some lemmas that will be needed in the course of our proof of Conjecture 1.3.

\begin{prop}
    There is a morphism $\omega\rightarrow\omega^{\text{Log}}$ from the Levine-Pandharipande algebraic cobordism ring to the logarithmic cobordism ring taking the class $(X)$ to the class $(X,\emptyset)$
\end{prop}

\begin{proof}

Write $\omega = \mathcal{M}^{'} / \mathcal{R}^{'}$. Suppose we have a double point degeneration $X \rightsquigarrow A\cup_D B$ given by $\mathcal{X} \rightarrow \mathbb{P}^1$, which induces the relation $[X] - [A] - [B] + [\mathbb{P}_D] \in \mathcal{R}^{'}$. Our goal is to show
\begin{equation*}
    [X,\emptyset] - [A,\emptyset] - [B,\emptyset] + [\mathbb{P}(N_{D/A}\oplus\mathcal{O}),\emptyset] \in \mathcal{R}
\end{equation*}
The degeneration $X \rightsquigarrow A\cup_D B$ tells us that 
\begin{equation*}
    a = [X,\emptyset] - [A,D] - [B,D]  \in \mathcal{R}
\end{equation*}
Next, degeneration to the normal cone for $D\subset A$ and degeneration to the normal cone for $D\subset B$ gives us the relations
\begin{align*}
    &b = [A,\emptyset] - [A,D] - [\mathbb{P}(N_{D/A}\oplus\mathcal{O}),D]  \in \mathcal{R}\\
    &c = [B,\emptyset] - [B,D] - [\mathbb{P}(N_{D/B}\oplus\mathcal{O}),D]  \in \mathcal{R}
\end{align*}
We now have
\begin{equation*}
    a - b - c = [X,\emptyset] - [A,D] - [B,D] + [\mathbb{P}(N_{D/A}\oplus\mathcal{O}),D] + [\mathbb{P}(N_{D/B}\oplus\mathcal{O}),D] \in \mathcal{R}
\end{equation*}
Using that $A$ and $B$ meet transversally along $D$, we get that $N_{A/D} = \check{N_{B/D}}$. Therefore, since $D = \mathbb{P}(N_{D/A})\subset \mathbb{P}(N_{D/A}\oplus\mathcal{O}) $ has normal bundle $\check{N_{A/D}} = N_{B/D}$, the degeneration to the normal cone of $D\subset \mathbb{P}(N_{D/A}\oplus\mathcal{O})$ gives us that 
\begin{equation*}
    (\mathbb{P}(N_{D/A}\oplus\mathcal{O},\emptyset) = (\mathbb{P}(N_{D/A}\oplus\mathcal{O},D) + \mathbb{P}(N_{D/B}\oplus\mathcal{O},D)
\end{equation*}
So we conclude that 
\begin{equation*}
    a - b - c = [X,\emptyset] - [A,D] - [B,D] + [\mathbb{P}(N_{D/A}\oplus\mathcal{O}),\emptyset] \in \mathcal{R}
\end{equation*}

\end{proof}

Next, we prove the following proposition which will be necessary when we wish to trade irreducible components of the boundary divisor for line bundles.

\begin{prop}
    There is a morphism $\omega^{\text{Log}}_{1^r}\rightarrow\omega^{\text{Log}}_{1^{r-1}}$ taking the class $(X,D,L_1,...,L_r)$ to the class $(\mathbb{P}(L_r\oplus \mathcal{O}), X\cup \eta^{-1}(D), \eta^\ast L_1,...,\eta^\ast L_{r-1})$ where $\eta: \mathbb{P}(L\oplus\mathcal{O})\rightarrow X$ is the bundle map
\end{prop}

\begin{proof}

Suppose we are given a degeneration of $[X,D,L_1,...,L_r]$. So we have $\mathcal{D} \subset \mathcal{X}\xrightarrow{\pi}\mathbb{P}^1$ satisfying the necessary properties, such that $X$ is a smooth fiber, $\pi^{-1}(0) = A\cup_E B$, and the total space of the degeneration carries line bundles $\mathcal{L}_1,...,\mathcal{L}_r\in\text{Pic}(\mathcal{X})$. From this data, we get an induced degeneration of $(\mathbb{P}(L_r\oplus \mathcal{O}), X\cup \eta^{-1}(D), \eta^\ast L_1,...,\eta^\ast L_{r-1})$. Indeed, consider 
\begin{equation*}
    \widetilde{\eta}: \mathbb{P}(\mathcal{L}_r\oplus\mathcal{O})\rightarrow \mathcal{X}
\end{equation*}
Also let
\begin{equation*}
    \widetilde{\mathcal{D}} = \mathcal{X}\cup\widetilde{\eta}^{-1}(\mathcal{D}),\quad  \widetilde{\mathcal{L}}_i = \widetilde{\eta}^\ast \mathcal{L}_i, \quad i=1,...,r-1
\end{equation*}
where $\mathcal{X}$ sits inside $\mathbb{P}(\mathcal{L}_r\oplus\mathcal{O})$ as $\mathbb{P}(\mathcal{L}_r)$. The data of $\pi\circ\widetilde{\eta}: \mathbb{P}(\mathcal{L}_r\oplus\mathcal{O})\rightarrow \mathbb{P}^1 $ together with the snc divisor $\widetilde{\mathcal{D}}$ and line bundles $\widetilde{\mathcal{L}}_1$,...,$\widetilde{\mathcal{L}}_{r-1}$ provides a degeneration of $(\mathbb{P}(L_r\oplus \mathcal{O}), X\cup \eta^{-1}(D), \eta^\ast L_1,...,\eta^\ast L_{r-1})$

\end{proof}

With these two propositions in hand we can now make progress towards proving Conjecture 1.3. Our aim is to prove 
\begin{theo}
$\omega^{\text{Log}}_3$ is generated over $\mathbb{Q}$ by smooth pairs and pairs with no boundary.
\end{theo}
To start off, we use our procedure of interchanging irreducible components for line bundles.
\begin{prop}
To prove $\omega^{\text{log}}_3$ is generated over $\mathbb{Q}$ by smooth pairs and pairs with no boundary it suffices to show that that any threefold $(\mathbb{P}(L\oplus \mathcal{O}), S\cup \eta^{-1}(D))$ obtained from $(S,D,L)$ can be written as a rational sum of smooth pairs and pairs with no boundary, where $S$ is a smooth surface, $D\subset S$ an snc divisor, and $L\in \text{Pic}(S)$.
\end{prop}

\begin{proof}
Let $(X,D)$ be a threefold snc pair, and suppose that $D = D_1\cup\dots D_k$ has $k$ irreducible components. By induction on the number $k$ of irreducible components, we can suppose that $(X,D_1\cup\dots\cup D_{k-1})$ can be written as a rational sum of pairs with either smooth boundary or no boundary. Consider the degeneration to the normal cone for $D_k$    
\begin{equation*}
    \mathcal{X} = \text{Bl}_{D_k\times 0} X\times \mathbb{P}^1\xrightarrow{\pi} \mathbb{P}^1
\end{equation*}
For our family of divisors we take the strict transform $\mathcal{D}$ of  $ (D_1\cup\dots \cup D_{k-1})\times \mathbb{P}^1$. From the degeneration $\mathcal{D}\subset\mathcal{X}\xrightarrow{\pi}\mathbb{P}^1$ we extract the relation 
\begin{equation*}
    (X,D_1\cup\dots D_{k-1}) = (X,D_1\cup\dots \cup D_{k-1}\cup D_k) + ((\mathbb{P}(L\oplus \mathcal{O}), D_k\cup \eta^{-1}(Z))
\end{equation*}
where $L$ is the normal bundle of $D_k$ inside of $X$, $D_k$ sits inside of $\mathbb{P}(L\oplus \mathcal{O})$ as $\mathbb{P}(L)$ and $Z = (D_1\cap D_k) \cup \dots\cup (D_{k-1}\cap D_k)$ is the boundary on $D_k$ induced from $X$. Rewriting the above equality, we get
\begin{equation*}
    (X,D_1\cup\dots \cup D_{k-1}\cup D_k) = (X,D_1\cup\dots D_{k-1}) - ((\mathbb{P}(L\oplus \mathcal{O}), D_k\cup \eta^{-1}(Z))
\end{equation*}
By induction on the number $k$ of irreducible components we know that $(X,D_1\cup\dots D_{k-1})$ can be written as a rational sum of pairs with either smooth boundary or no boundary, hence if we wish to prove $(X,D_1\cup\dots \cup D_k)$ can be expressed as a rational sum of pairs with smooth or no boundary it suffices to prove it for $((\mathbb{P}(L\oplus \mathcal{O}), D_k\cup \eta^{-1}(Z))$
\end{proof}
The upshot of the above proposition is that $((\mathbb{P}(L\oplus \mathcal{O}), D_k\cup \eta^{-1}(Z))$ is the image of the \textit{surface} snc line bundle tuple $(D_k, Z, L)$ under the morphism in Proposition 3.5. 

Hence, we now move on to proving the following proposition about threefolds obtained from surface snc pairs with a line bundle. Note that an snc pair with two line bundles $(C, D, L,M) $ given by a smooth curve $C$, an snc divisor $D$ on $C$, and two line bundles $L,M\in \text{Pic}(C)$ induces a threefold pair by applying Proposition 3.3 twice. The resulting threefold snc pair is 
\begin{equation*}
    (\mathbb{P}(\eta^\ast L\oplus\mathcal{O}), \mathbb{P}(M\oplus \mathcal{O})\cup (\eta\circ\widetilde{\eta})^{-1}(D)) 
\end{equation*}
where $\widetilde{\eta}: \mathbb{P}(\eta^\ast L\oplus\mathcal{O}) \rightarrow \mathbb{P}(M\oplus \mathcal{O})$ and $\eta: \mathbb{P}(M\oplus \mathcal{O})\rightarrow C $ are the bundle maps.
\begin{prop}

To prove any threefold snc pair $(\mathbb{P}(L\oplus \mathcal{O}), S\cup \eta^{-1}(D))$ induced by an snc surface pair with line bundle $(S,D,L)$ can be written as a rational sum of terms with either smooth boundary or no boundary, it suffices to prove that any threefold snc pair obtained from a curve snc pair with two line bundles $(C, D, L,M) $ can be written as a rational sum of terms with either smooth or no boundary.
\end{prop}

\begin{proof}
As in the proof of proposition 3.5, we induct on the number $k$ of irreducible components for $D = D_1\cup \dots \cup D_k$. When $k = 0$, $(S,\emptyset, L)\mapsto (\mathbb{P}(L\oplus\mathcal{O}),S)$ indeed gives a smooth threefold pair. By induction, $(S,D_1\cup\dots\cup D_{k-1},L)$ gives a threefold pair that can be written as a rational sum of terms with either empty or smooth boundary. 

Consider the degeneration to the normal cone for $D_k$    
\begin{equation*}
    \mathcal{S} = \text{Bl}_{D_k\times 0} S\times \mathbb{P}^1\xrightarrow{\pi} \mathbb{P}^1
\end{equation*}
Take the line bundle $\mathcal{L}$ to be $\text{pr}^\ast_1\circ \text{bl}^\ast L$ where $\text{bl}: \mathcal{S}\rightarrow S\times\mathbb{P}^1$ is the blow up map and $\text{pr}_1; S\times\mathbb{P}^1\rightarrow S$ is the projection. For our family of divisors we take the strict transform $\mathcal{D}$ of  $ (D_1\cup\dots \cup D_{k-1})\times \mathbb{P}^1$. From the degeneration $\mathcal{D}\subset\mathcal{S}\xrightarrow{\pi}\mathbb{P}^1$ we extract the relation 
\begin{equation*}
    (S,D_1\cup\dots D_{k-1},L) = (S,D_1\cup\dots \cup D_{k-1}\cup D_k,L) + ((\mathbb{P}(M\oplus \mathcal{O}), D_k\cup \eta^{-1}(Z),\eta^\ast L)
\end{equation*}
where $M$ is the normal bundle of $D_k$ inside of $S$, $D_k$ sits inside of $\mathbb{P}(M\oplus \mathcal{O})$ as $\mathbb{P}(M)$ and $Z = (D_1\cap D_k) \cup \dots\cup (D_{k-1}\cap D_k) = Z_1\cup\dots Z_{k-1}$. By our inductive hypothesis it now suffices to prove that $((\mathbb{P}(M\oplus \mathcal{O}), D_k\cup \eta^{-1}(Z),\eta^\ast L)$ can be written as a sum of pairs with either empty or smooth boundary. Note that by Proposition 3.3 to degenerate $((\mathbb{P}(M\oplus \mathcal{O}), D_k\cup \eta^{-1}(Z),\eta^\ast L)$ it suffices to degenerate $(D_k, Z, L,M)$ where now $D_k$ is a smooth curve.

\end{proof}

We can now conclude with the following 
\begin{prop}
Given a smooth curve $C$ with an snc divisor $D$ and two line bundles $L,M\in\text{Pic}(C)$, we have an induced threefold snc pair 
\begin{equation*}
    (\mathbb{P}(\eta^\ast L\oplus\mathcal{O}), \mathbb{P}(M\oplus \mathcal{O})\cup (\eta\circ\widetilde{\eta})^{-1}(D)) 
\end{equation*}
where $\widetilde{\eta}: \mathbb{P}(\eta^\ast L\oplus\mathcal{O}) \rightarrow \mathbb{P}(M\oplus \mathcal{O})$ and $\eta: \mathbb{P}(M\oplus \mathcal{O})\rightarrow C $ are the bundle maps.

Then, any such threefold snc pair can be written as a rational sum of terms with either no boundary or smooth boundary.

\end{prop}

\begin{proof}
We will prove the theorem by induction on the number of points $k$ in the snc divisor $D = p_1 + \dots + p_k$ on the smooth curve $C$. We begin with the case $k=0$. So we have $(C,\emptyset, L,M)\in \omega^{\text{Log}}_{1,1^2}$ where $C$ is a smooth curve and $L,M$ are line bundles on $C$. We now use the morphism $\omega_{1,1^2}\rightarrow \omega^{\text{Log}}_{1,1^2}$ from the Levine-Pandharipande ring to our ring. In particular, it is proven by Lee-Pandharipande in \cite{LeeP} that $\omega_{1,1^2}$ has a $\mathbb{Q}$ basis given by $(\mathbb{P}^1,\mathcal{O}(1),\mathcal{O})$, $(\mathbb{P}^1, \mathcal{O},\mathcal{O}(1))$ and $(\mathbb{P}^1,\mathcal{O},\mathcal{O})$. The element $(\mathbb{P}^1,\emptyset,\mathcal{O}(1),\mathcal{O})$ induces the smooth threefold pair
\begin{equation*}
    (\mathbb{F}_1\times\mathbb{P}^1, \mathbb{F}_1\times 0\cup \mathbb{P}^1\times \mathbb{P}^1) = (\mathbb{F}_1,\mathbb{P}^1)\cdot (\mathbb{P}^1,0) = \frac{1}{2}(\mathbb{F}_1\times\mathbb{P}^1,\mathbb{P}^1\times\mathbb{P}^1)
\end{equation*}
where we used the fact that $(\mathbb{P}^1,0) = \frac{1}{2}(\mathbb{P}^1,\emptyset)$. The element $(\mathbb{P}^1, \emptyset, \mathcal{O},\mathcal{O}(1))$ induces the same smooth threefold pair, and the element $(\mathbb{P}^1,\emptyset,\mathcal{O},\mathcal{O})$ induces the smooth threefold pair
\begin{equation*}
    ((\mathbb{P}^1)^2\times \mathbb{P}^1, (\mathbb{P}^1\times0)\times\mathbb{P}^1\cup (\mathbb{P}^1)^2\times 0) = (\mathbb{P}^1,0)^3 = \frac{1}{8} ((\mathbb{P}^1)^3,\emptyset)
\end{equation*}
Using the morphism $\omega_{1,1^2}\rightarrow \omega^{\text{Log}}_{1,1^2}$ we know then that $(C,\emptyset, L,M)$ can be written as a $\mathbb{Q}$ linear combination of $(\mathbb{P}^1,\mathcal{O}(1),\mathcal{O})$, $(\mathbb{P}^1, \mathcal{O},\mathcal{O}(1))$, and so we conclude by the above discussion that $(C,\empty,L,M)$ can be written as a sum of pairs with either empty or smooth boundary.

For the inductive step, write $D = p_1 + \dots + p_k$, and consider the degeneration to the normal cone for $p_k$
\begin{equation*}
    \mathcal{C} = \text{bl}_{p_k\times 0}(C\times\mathbb{P}^1)\xrightarrow{\text{bl}} C\times\mathbb{P}^1
\end{equation*}
Let $\text{pr}_1: C\times\mathbb{P}^1\rightarrow C$ be the projection and let $\mathcal{L}= (\text{pr}_1\circ \text{bl})^\ast L$ and $\mathcal{M} =  (\text{pr}_1\circ \text{bl})^\ast M$. For the family of divisors let $\mathcal{D}\subset \mathcal{C}$ be the strict transform of ${p_1}\times\mathbb{P}^1\cup\dots \cup {p_{k-1}}\times\mathbb{P}^1$. The relation we extract from $\mathcal{D}\subset\mathcal{C}\xrightarrow{\pi}\mathbb{P}^1$ is 
\begin{equation*}
    (C,p_1+\dots + p_{k-1}, L,M) = (C_,p_1\dots + p_{k}, L,M) + (\mathbb{P}^1, p_k, \mathcal{O},\mathcal{O})
\end{equation*}
By induction we know that the threefold induced by $(C,p_1+\dots + p_{k-1}, L,M)$ can be written as a $\mathbb{Q}$ sum of terms with either no boundary or smooth boundary. We are now done, because $(\mathbb{P}^1, p_k, \mathcal{O},\mathcal{O})$ induces the following pair with no boundary
\begin{equation*}
    ((\mathbb{P}^1)^3, 0\times(\mathbb{P}^1)^2\cup (\mathbb{P}^1)^2\times 0\cup \mathbb{P}^1\times 0\times\mathbb{P}^1) = (\mathbb{P}^1,0)^3 = \frac{1}{8}((\mathbb{P}^1)^3,\emptyset)
\end{equation*}

\end{proof}

Combining Propositions 3.5, 3.6, 3.7, we have proven Theorem 3.4.

Note that by work of Lee-Pandharipande \cite{LeeP} the algebraic cobordism ring $\omega_{2,1^1}$ of surfaces and line bundles $(S,L)$ has a basis over $\mathbb{Q}$ given by $(\mathbb{P}^2, \mathcal{O})$, $(\mathbb{P}^2, \mathcal{O}(1))$, $(\mathbb{P}^1\times\mathbb{P}^1, \mathcal{O})$, and $(\mathbb{P}^1\times\mathbb{P}^1, \mathcal{O}(1,0))$. Applying the morphism of Proposition 3.3 to get elements of $\omega^{\text{Log}}_3$, the above elements of $\omega_{2,1^1}$ induce the following threefold snc pairs:
\begin{align*}
    &(\mathbb{P}^2, \mathcal{O})\mapsto (\mathbb{P}^2\times\mathbb{P}^1, \mathbb{P}^2\times 0) = (\mathbb{P}^2,\emptyset)\cdot (\mathbb{P}^1,0) = \frac{1}{2}(\mathbb{P}^2\times\mathbb{P}^1,\emptyset)\\
    &(\mathbb{P}^2, \mathcal{O}(1))\mapsto (\mathbb{P}(\mathcal{O}_{\mathbb{P}^2}(1)\oplus\mathcal{O}),\mathbb{P}^2)\\
    &(\mathbb{P}^1\times\mathbb{P}^1, \mathcal{O})\mapsto ((\mathbb{P}^1)^3, \mathbb{P}^1\times\mathbb{P}^1\times0) = (\mathbb{P}^1\times\mathbb{P}^1,\emptyset)\cdot (\mathbb{P}^1,0) = \frac{1}{2}((\mathbb{P}^1)^3,\emptyset0\\
    &(\mathbb{P}^1\times\mathbb{P}^1, \mathcal{O}(1,0))\mapsto (\mathbb{F}_1\times\mathbb{P}^1, \mathbb{P}^1\times\mathbb{P}^1)
\end{align*}
Note that any smooth threefold pair $(X,D)$ can be written as $(X,\emptyset) - (D,N_{D/X})$ by performing deformation to the normal cone for $D\subset X$. In particular, any smooth threefold pair can be written as a rational sum of the images of the generators of $\omega_3\rightarrow \omega^{\text{Log}}_3$, i.e. terms without boundary, and the images of the generators of the composition $\omega_{2,1^1}\rightarrow \omega^{\text{Log}}_{2,1^1}\rightarrow \omega^{\text{Log}}_{3}$. The ring $\omega_3$ has a basis over $\mathbb{Q}$ given by $(\mathbb{P}^3)$, $(\mathbb{P}^2\times\mathbb{P}^1), (\mathbb{P}^1\times\mathbb{P}^1\times\mathbb{P}^1)$, so in conjunction with the above discussion and using Theorem 3.4 we conclude

\begin{theo}
$\omega^{\text{Log}}_3$ is generated as a vector space over $\mathbb{Q}$ be the elements $(\mathbb{P}^3,\emptyset)$, $(\mathbb{P}^2\times\mathbb{P}^1,\emptyset), (\mathbb{P}^1\times\mathbb{P}^1\times\mathbb{P}^1,\emptyset)$, $(\mathbb{P}(\mathcal{O}_{\mathbb{P}^2}(1)\oplus\mathcal{O}),\mathbb{P}^2)$, and $(\mathbb{F}_1\times\mathbb{P}^1, \mathbb{P}^1\times\mathbb{P}^1)$
\end{theo}

\begin{rem}
The element $(\mathbb{P}(\mathcal{O}_{\mathbb{P}^2}(1)\oplus\mathcal{O}),\mathbb{P}^2)$ in the above list can be replaced by $(\mathbb{P}^3,\mathbb{P}^2)$. This can be seen by performing degeneration to the normal cone for a hyperplane $\mathbb{P}^2\subset \mathbb{P}^3$.
\end{rem}

\section{Proof of the zero dimensional log DT generating funfction}

Having generators for $\omega^{\text{Log}}_3$ in hand, we are nearly able to prove Conjecture 1.3. The following gives us invariants for logarithmic cobordism
\begin{theo}
For any partition $\lambda$ of $n$, there is a morphism $c_\lambda: \omega^{\text{Log}}_n\rightarrow \mathbb{Q}$ of $\mathbb{Q}$ vector spaces $\omega^{\text{Log}}_n\rightarrow \mathbb{Q}$ defined by 
\begin{equation*}
    (X,D)\mapsto c^{\text{log}}_\lambda(X,D) = \prod_i c_{\lambda_i} (T_{(X,D)})
\end{equation*}
where $T_{(X,D)}$ is the logarithmic tangent bundle with respect to the snc divisor $D\subset X$ and $\lambda_i$ are the parts of $\lambda$
\end{theo}

\begin{proof}
Before beginning the proof we briefly recall one property about the log tangent bundle that we will use. Suppose $\partial X =  D_1\cup\dots\cup D_k \subset X$ is an snc divisor in a smooth variety $X$ where each $D_i$ is smooth. Let $Z\subset X$ be a smooth subvariety meeting $\partial X$ transversally, then the log tangent bundle $T^\text{log}_X$ of $X$ with respect to $D$ restricted to $Z$ is the log tangent bundle $T^{\text{log}}_{Z}$ of $Z$ with respect to the snc divisor $\partial Z = \partial X\cap Z = (D_1\cap Z)\cup\dots\cup (D_k\cap Z)$.

Now, suppose we have a double point degeneration of snc pairs $\mathcal{D}\subset \mathcal{X}\xrightarrow{\pi} \mathbb{P}^1$, with $\pi^{-1}(0) = A\cup_E B$ and $X = \pi^{-1}(\xi)$ a smooth fiber. Let $\partial X = X\cap \mathcal{D}$, $\partial A = (\mathcal{D}\cap A)\cup E = D_A\cup E$, and $\partial B = (\mathcal{D}\cap B)\cup E = D_B\cup E$. The double point relation extracted from the degeneration is then 
\begin{equation*}
    (X,\partial X) = (A,\partial A) + (B,\partial B)
\end{equation*}
The log tangent bundle $T^{\text{log}}_\mathcal{X}$ of $\mathcal{X}$ with respect to $\partial \mathcal{X} = \mathcal{D}$ restricted to $X$ is the log tangent bundle $T^{\text{log}}_X$ of $X$ with respect to $\partial X$. Similarly, the log tangent bundle $T^{\text{log}}_\mathcal{X}$ restricted to $A$ and $B$ gives $T^{\text{log}}_{(A,D_A)}$, i.e. the log tangent bundle of $A$ w.r.t. $D_A$, and the log tangent bundle $T^{\text{log}}_B$ of $B$ with respect to $D_B$. Let $\lambda$ be a partition of $n = \dim X$, then 
\begin{equation*}
    c_\lambda (T^{\text{log}}_X) = c_\lambda(T^{\text{log}}_\mathcal{X}|_X) = c_\lambda(T^{\text{log}}_\mathcal{X})\cdot [X]\in \text{CH}_0(\mathcal{X})
\end{equation*}
Using $[X] = [\pi^{-1}(\xi)] = [\pi^{-1}(0)] = [A] + [B]$ in $\text{CH}(\mathcal{X})$, we conclude that 
\begin{align*}
    c_\lambda (T^{\text{log}}_X) &= c_\lambda(T^{\text{log}}_\mathcal{X})\cdot [X] \\
    &= c_\lambda(T^{\text{log}}_\mathcal{X})\cdot ([A] + [B]) \\
    &= c_\lambda(T^{\text{log}}_\mathcal{X}|_A) + c_\lambda(T^{\text{log}}_\mathcal{X}|_B)  \\
    &= c_\lambda (T^{\text{log}}_{(A,D_A)} )+ c_\lambda (T^{\text{log}}_{(B,D_B)})
\end{align*}
It then remains to observe that 
\begin{equation*}
    c_\lambda (T^{\text{log}}_{(A,D_A)} )+ c_\lambda (T^{\text{log}}_{(B,D_B)}) = c_\lambda(T^{\text{log}}_{(A,D_A+E)}) + c_\lambda(T^{\text{log}}_{(B,D_B+E)})
\end{equation*}
This can be seen by combining the short exact sequences 
\begin{align*}
    &0\rightarrow T^{\text{log}}_{(A,D_A+E)}\rightarrow T^{\text{log}}_{(A,D_A)}\rightarrow N_{E\subset A}\rightarrow 0\\
    &0\rightarrow T^{\text{log}}_{(B,D_B+E)}\rightarrow T^{\text{log}}_{(B,D_B)}\rightarrow N_{E\subset B}\rightarrow 0
\end{align*}
and also using that since $A$ and $B$ meet transversally along $E$ then $N_{E\subset A} = N^{\vee}_{E\subset B}$, so 

\end{proof}

We now observe that for $(X,D)$ an snc pair, the quantity $\alpha(X,D) = c_3(T^{\text{log}}_X\otimes K^{\text{log}}_X)$ can be expressed as an integer sum of the log chern invariants $c^{\text{log}}_\lambda(\cdot,\cdot)$ for $\lambda$ a partition of 3. Hence $\alpha(X,D) = c_3(T^{\text{log}}_X\otimes K^{\text{log}}_X)$ is an invariant and yields a morphism of $\mathbb{Q}$ vector spaces $\alpha(\cdot,\cdot): \omega^{\text{Log}}_3\rightarrow\mathbb{Q}$. We can now prove our main theorem
\begin{theo}
For $(X,D)$ an snc threefold pair, we have an equality of generating functions
\begin{equation*}
    Z(X,D) = M(q)^{\int_X c_3(T^\text{log}_X\otimes K^\text{log}_X)}
\end{equation*}
\end{theo}

\begin{proof}

By Theorem 3.1 and Theorem 3.8, we can find integers $r,a_i$ such that the element $(X,D)\in \omega^{\text{Log}}_3$ in the logarithmic cobordism ring satisfies
\begin{equation*}
r(X,D) = a_1(\mathbb{P}^3,\emptyset) +a_2 (\mathbb{P}^2\times\mathbb{P}^1,\emptyset) + a_3(\mathbb{P}^1\times\mathbb{P}^1\times\mathbb{P}^1,\emptyset) + a_4(\mathbb{P}^3,\mathbb{P}^2) + a_5(\mathbb{F}_1\times\mathbb{P}^1, \mathbb{P}^1\times\mathbb{P}^1)
\end{equation*}
By Proposition 2.1, we have a ring homomorphism $\omega^{\text{Log}}_3\rightarrow \mathbb{Q}[[x]]$ sending $(X,D)$ to the zero dimensional log DT generating function $\text{Z}(X,D)$. Using the above equality we now have 
\begin{equation*}
    \text{Z}(X,D)^r = \text{Z}(\mathbb{P}^3,\emptyset)^{a_1} \cdot \text{Z}(\mathbb{P}^2\times\mathbb{P}^1,\emptyset)\cdot  \text{Z}(\mathbb{P}^1\times\mathbb{P}^1\times\mathbb{P}^1,\emptyset)^{a_3} \cdot \text{Z}(\mathbb{P}^3,\mathbb{P}^2)^{a_4} \cdot  \text{Z}(\mathbb{F}_1\times\mathbb{P}^1, \mathbb{P}^1\times\mathbb{P}^1)^{a_5}
\end{equation*}
Conjecture 1.3 has been proven for pairs with no boundary and for smooth pairs by Levine-Pandharipande in \cite{LP}, so the previous equality gives 
\begin{align*}
\text{Z}(X,D)^r &= M(-q)^{a_1\alpha(\mathbb{P}^3,\emptyset) + a_2 \alpha(\mathbb{P}^2\times\mathbb{P}^1,\emptyset) + a_3\alpha(\mathbb{P}^1\times\mathbb{P}^1\times\mathbb{P}^1,\emptyset) + a_4\alpha(\mathbb{P}^3,\mathbb{P}^2) + a_5\alpha(\mathbb{F}_1\times\mathbb{P}^1, \mathbb{P}^1\times\mathbb{P}^1)} \\
&= M(-q)^{r\alpha(X,D)} \\
&= M(-q)^{r c_3(T^{\text{log}}_X\otimes K^{\text{log}}_X)}
\end{align*}
Since $\text{Z}(X,D)(0) = 1$ and $M(0) =1$, we conclude 
\begin{equation*}
    \text{Z}(X,D) = M(-q)^{c_3(T^{\text{log}}_X\otimes K^{\text{log}}_X)}
\end{equation*}

\end{proof}

\section{Studying the logarithmic cobordism ring}

In this section we describe all invariants of the logarithmic cobordism ring, which are morphisms of $\mathbb{Q}$ rings $\omega^{\text{log}}_n\rightarrow \mathbb{Q}$. Recall that the notation $\lambda \vdash n$ means that $\lambda$ is a partition of $n$.

\begin{theo}
Let $(X,D)$ be a simple normal crossings pair with $\dim X = n$. Choose $i,k,\lambda \vdash (n-(i+1)k))$ such that $i$ is odd and $n-(i+1)k\geq0$, then 

\begin{equation*} 
\alpha_{i,\lambda,k}(X,D) \coloneqq \sum_{V_k\subset X, \text{codim $k$ strata}} c_k(N_{V/X})^i c_\lambda(T^{\text{log}}_V)
\end{equation*}
where the log structure of $V$ is inherited from $X$. Then $\alpha_{i,k,\lambda}$ gives us an invariant $\omega^{\text{log}}_n\rightarrow \mathbb{Q}$
\end{theo}

\begin{proof}
All we need to show is that $\alpha_{i,k,\lambda}$ respects double point degenerations. Thus, suppose $\mathcal{D}\subset \mathcal{X}\rightarrow \mathbb{P}^1$ is a degeneration of snc pairs as in our definition, let $(X,D)$ be a smooth fiber, where $D = \mathcal{D}\cap X$ and let the singular fiber be $A\cup_E B$, so the log structures on the components of the singular fiber are $(A,E\cup D_A)$ and $(B,E\cup D_B)$, where $D_A = A\cap\mathcal{D}$ and $D_B = B\cap \mathcal{D}$. Consider $\mathcal{X}$ with boundary given by $\mathcal{D}$, and consider the following chow class in $X$
\begin{equation*} 
\sum_{\mathcal{V}_k \subset \mathcal{X},\text{codim $k$ strata}} c^{k}(N_{\mathcal{V}_k\subset \mathcal{X}})^ic_{\lambda}(T^{log}_{\mathcal{V}_k}) \in \text{CH}(\mathcal{X})
\end{equation*}
Since $\mathcal{X}\rightarrow \mathbb{P}^1$ is flat, we know $[X] = [A\cup_E B] = [A] + [B]$ in $\text{CH}(\mathcal{X})$, and so we conclude that intersecting the above class with $[X] = [A] +[B]$ gives us the equality 
\begin{align*}
\sum_{V_k\subset X} c_k(N_{V/X})^i c_\lambda(T^{\text{log}}_V) =  \sum_{V^A_k\subset A, V^A_k\not \subset E } c^i_k(N_{V^A\subset A})c_\lambda(T^\text{log}_{V^A_k}) + \sum_{V^B_k\subset B,V^B_k\not\subset E} c^i_k(N_{V^B\subset B})c_\lambda(T^{log}_{V^B_k})
\end{align*}
where all sums are over strata of codimension $k$. The left hand side of the equation is equal to $\alpha_{i,\lambda,k}(X,D)$, and right hand side of the above equation is indeed equal to $\alpha_{i,\lambda,k}(A,D_A\cup E)+\alpha_{i,\lambda,k}(B,D_B\cup E))$. The reason is that the terms coming from strata being contained in $E$ cancel out due to the fact that $N_{E/A} = N^\vee_{E/B}$. So we get that $\alpha_{i,\lambda,k}(X,D) = \alpha_{i,\lambda,k}(A,D_A\cup E) + \alpha_{i,\lambda,k}(B,D_B \cup E)$.

\end{proof}

We will prove that the above invariants $\alpha_{i,\lambda,k}$ are all of the invariants of the logarithmic cobordism ring $\omega^{\text{log}}_n$. To understand the invariants of $\omega^{\text{log}}_n$ it suffices to understand them on generators of the ring, and here we describe a useful set of generators. Recall that there are maps $\eta_k:\omega_{k,1^{n-k}}\rightarrow \omega^{\text{log}}_n$ which send a variety with line bundles $(Z,L_1,\dots,L_k)$ to an snc pair by iteratively taking $\mathbb{P}^1$ bundles. Then this list of morphisms gives us additive generators for the log cobordism ring
\begin{equation*}
\omega^{\text{log}}_n = \sum_{k=1,n-1} \eta_k(\omega_{k,1^{n-k}})
\end{equation*}
This is not hard to seen and we have implicitly proved this result already via our process of trading divisor components for line bundles. For clarity we spell this out here. Suppose that $(X,D)$ is an snc pair with $\dim X = n$. Write $D = D_1 \cup \dots \cup D_k$. Then by degenerating to the normal cone of $D_1$ we get 
\begin{equation*}
(X,(D_2\cup \dots \cup D_{k-1})) = (X,D_1\cup \dots \cup D_k) + (D_1,Z,N_{D_1/X})
\end{equation*}
By induction on the number of irreducible components $k$, we can assume that the left hand side in the above equation can be written as a sum of terms of the form $\eta_i(W,L_1,...,L_i)$. Can continue to iterate and eventually conclude that $(D,Z,L)$ can also be written as a sum of varieities with no boundary and only line bundles, and the reason for this is that we can keep trading boundary components as in the above step until we have pushed off all of the boundary components and we only have line bundles.
\\
\\
With this in hand, we can now prove that the invariants we described above are all of the invariants. 
\begin{theo}
Consider an invariant $\phi:\omega^{\text{log}}_n\rightarrow \mathbb{Q}$. Then $\phi$ is in the span of the $\alpha_{i,\lambda,k}$
\end{theo}
\begin{proof}
First, we understand what $\phi$ restricted to pairs $(X,\emptyset)$ with no boundary is: it must be a linear combination of the chern invariants $c_\lambda$, so we may assume that $\phi$ is zero on pairs with no boundary by substracting off this restriction. Next, let us see what $\phi$ does on pairs $(X,D)$ where $D$ is a smooth divisor, or equivalently we need to understand what $\phi$ does on pairs $(D,L)$, $L\in\text{Pic}(L)$, where this corresponds to the pair $(\mathbb{P}(L\oplus \mathcal{O}),\mathbb{P}(L))$ with $D = \mathbb{P}(L)$. By degenerating along the normal cone for $\mathbb{P}(L)\subset \mathbb{P}(L\oplus \mathcal{O})$ we get the following relation:
\begin{equation*}
(\mathbb{P}(L\oplus O),\emptyset) = (\mathbb{P}(L\oplus \mathcal{O}),\mathbb{P}(L)) + (\mathbb{P}(L^\vee\oplus \mathcal{O}),\mathbb{P}(L^\vee))
\end{equation*}
since the normal bundle of $D = \mathbb{P}(L)$ inside of $\mathbb{P}(L\oplus \mathcal{O})$ is $L^\vee$. Applying $\phi$ to both sides and using that $\phi$ is zero on pairs with no boundary, we get
\begin{equation*}
    \phi(D,L) = -\phi(D,L^\vee)
\end{equation*}
The invariants of $(D,L)$, by the work of Lee-Pandharipande are of the form $c_\lambda(D)c_1(L)^i$, and the above equality tells us that $i$ must be odd. This allows us to completely understand what $\phi$ does on pairs with smooth boundary, i.e. on the elements of our generating set previously described of the form $(D,L)$. Now consider elements of the form $(Z,L,M)$. Then we claim that on these elements of the generating set we must get that $\phi$ satisfies 
\begin{align*}
    &\phi(Z,L,M) = \phi(Z,M,L),\\
    &\phi(Z,L^{-1},M) = -\phi(Z,L,M)\\
    &\phi(Z,L,M^{-1}) = -\phi(Z,L,M)
\end{align*}
Each of these can be seen via appropriate degenerations to the normal cone. The only invariants of $(Z,L,M)$ which satisfy this condition are those of the form 
\begin{equation*}
    c_\lambda(Z)c_1(L)^ic_1(M)^i
\end{equation*}
Continuing on in this way, we are able to determine $\phi$ when restricted to each piece of the generating set
\begin{equation*}
\omega^{\text{log}}_n = \sum^{n-1}_{k=1} \eta_k(\omega_{k,1^{n-k}})
\end{equation*}
In general, $\phi$ restricted to elements of the form $(Z,L_1,\dots L_s)$ must correspond to elements of the form $c_\lambda(Z)c_1(L_1)^{i}\dots c_1(L_s)^i$ where $i$ is odd. We then conclude that $\phi$ is in the span of our invariants, because each of the invariants $\alpha_{i,\lambda,k}$ will restrict to a sum of the invariants $c_\lambda(Z)c_1(L_1)^{i}\dots c_1(L_s)^i$ on elements of the form $(Z,L_1,\dots,L_s)$, and we have just shown that these determine $\phi$.
\end{proof}

\section{Logarithmic cobordism with log modifications}

Let $(X,D)$ be a simple normal crossings pair, and recall that a logarithmic modification $(X',D')$ of $(X,D)$ is a blow up of $(X,D)$ along strata, with $D' = \widetilde{D}\cup E$ being the strict transform of $D$ along with the exceptional divisor. Let $\omega^{\text{Log+Mod}}$ be the quotient of $\omega^{\text{Log}}$ where we further impose the relation $(X',D') = (X,D)$ for any log modification $(X',D')\rightarrow (X,D)$. For example, suppose we have $(\mathbb{P}^2,L_1\cup L_2)$ where $L_1,L_2$ are two lines in $\mathbb{P}^2$ meeting transversally, then we have the equality $(\text{Bl}_p(\mathbb{P}^2),\widetilde{L}_1\cup \widetilde{L}_2\cup E) = (\mathbb{P}^2,L_1\cup L_2)$ in $\omega^{\text{Log+Mod}}$. 
We now have the following diagram

\begin{center}
\begin{tikzcd}
\omega^{LP} \arrow[d] \arrow[rd] &                    \\
\omega^{Log} \arrow[r]           & \omega^{Log + Mod}
\end{tikzcd}
\end{center}

and the goal of this section is to prove that the map $\omega^{LP}\rightarrow \omega^{\text{Log+Mod}}$ is an isomorphism. Before we begin proving this, we start off with some observations that we will use.

Let $(X,D)$ be an snc pair and let $Z = D_{i_1}\cap \dots \cap D_{i_k}$ be a stratum of $X$. Consider the total space $\mathcal{X} = \text{Bl}_{Z\times 0}(X\times \mathbb{P}^1)\rightarrow \mathbb{P}^1$ which gives us a degeneration $X\rightsquigarrow \text{Bl}_Z(X)\cup \mathbb{P}(N_Z\oplus \mathcal{O})$. Taking the strict transforms $\widetilde{\mathcal{D}} \subset \mathcal{X}$ as our family of divisors, where $\mathcal{D} = D\times\mathbb{P}^1$, then the deformation to the normal cone for $Z\subset X$ gives us the relation
\begin{equation*}
    (X,D) = (\text{Bl}_Z(X), D') + (\mathbb{P}(N_Z\oplus \mathcal{O}),\eta^{-1}(\partial Z)\cup \mathbb{P}(N_Z)\cup \mathbb{P}_0\cup \dots \cup \mathbb{P}_k)
\end{equation*}
where by $\mathbb{P}_i$ we mean the following. First, the normal bundle $N_Z$ splits as $L_1\oplus\dots\oplus L_k$ since $Z$ is the transverse intersection of boundary divisors, and $\mathbb{P}_s = \mathbb{P}(N_Z/L_{i_s}\oplus \mathcal{O})$ for $s\geq1$ and $\mathbb{P}_0 = \mathbb{P}(N_Z) = \mathbb{P}(L_1\oplus \dots L_k)$. More geometrically, the term snc pair $(\mathbb{P}(N_Z\oplus \mathcal{O}),\eta^{-1}(\partial Z)\cup \mathbb{P}(N_Z)\cup \mathbb{P}_1\cup \dots \cup \mathbb{P}_k)$ has a map to the base $(Z,\partial Z)$, and the given boundary on $\mathbb{P}(N_Z\oplus \mathcal{O})$ restricted to a fiber $F\cong \mathbb{P}^k$ endows $\mathbb{P}^k$ with its full toric boundary consisting of the $k+1$ hyperplanes $H_0\cup \dots \cup H_k$. 

What the above observation allows us to conclude is that a log modification $(\text{Bl}_Z(X),D')$ of $(X,D)$ is equal to $(X,D)$ up to a 'correction' term, which consists of a projective space bundle $\mathbb{P}\rightarrow Z$. So in $\omega^{\text{Log+Mod}}$ we have:
\begin{equation*}
(\mathbb{P}(N_Z\oplus \mathcal{O}),\eta^{-1}(\partial Z)\cup \mathbb{P}(N_Z)\cup \mathbb{P}_0\cup \dots \cup \mathbb{P}_k) = 0 \in \omega^{\text{Log+Mod}}
\end{equation*}
where $Z\subset X$ is a stratum in some snc pair $(X,D)$.
In particular, consider the case in which our stratum $Z$ is a point, i.e. $Z$ is the intersection of $n = \dim X$ boundary divisors, then the above observation tells us
\begin{equation*}
    (\mathbb{P}^n,H_0\cup\dots\cup H_n) = 0\in\omega^{\text{Log+Mod}}
\end{equation*}
Somewhat surprisingly, the main ingredient in showing $\omega^{\text{LP}}\rightarrow \omega^{\text{Log+Mod}}$ that is an isomorphism is the relation $(\mathbb{P}^n,H_0\cup\dots\cup H_n) = 0$. 

Before we embark on our proof, we present one more simple observation. Suppose that $(X,L)$ and $(Y,M)$ are smooth projective varieties with two line bundles $L\in \text{Pic}(X)$, $M\in\text{Pic}(Y)$. Recall that we have a process that associates snc pairs to any variety with line bundles: 
\begin{equation*}
    (X,L)\mapsto (\mathbb{P}(L\oplus \mathcal{O}),\mathbb{P}(L) = X),\quad (Y,M)\mapsto (\mathbb{P}(Y\oplus \mathcal{O}),\mathbb{P}(M) = Y),\quad 
\end{equation*}
Recall that we also have a product of snc pairs in our logarithmic cobordism ring:
\begin{equation*}
    (Z,\partial Z)\times (W,\partial W) = (Z\times W,\partial Z\times W\cup Z\times \partial W)
\end{equation*}
Then the following is true:
\begin{prop}
The product of snc pairs $(\mathbb{P}(L\oplus O),\mathbb{P}(L))\cdot (\mathbb{P}(M\oplus \mathcal{O}),\mathbb{P}(M))$ is the same as the snc pair associated to $(X\times Y, \text{pr}^\ast_XL,\text{pr}^\ast_Y M)$
\end{prop}
\begin{proof}
This follows from the elementary observation that the snc pair associated to $(X\times Y,\text{pr}^\ast_XL)$ is $(\mathbb{P}(L\oplus \mathcal{O})\times Y, \mathbb{P}(L)\times Y)$; so the snc pair associated to $(\mathbb{P}(L\oplus \mathcal{O})\times Y, \mathbb{P}(L)\times Y,\text{pr}^\ast_YM)$ is $(\mathbb{P}(L\oplus \mathcal{O})\times \mathbb{P}(M\oplus \mathcal{O}),\mathbb{P}(L)\times  \mathbb{P}(M\oplus\mathcal{O})\cup \mathbb{P}(L\oplus \mathcal{O})\times \mathbb{P}(M)) $
\end{proof}
We are now ready to begin the proof of our main theorem. 

\begin{theo}
Consider the diagram 
\begin{center}
\begin{tikzcd}
\omega^{LP} \arrow[d] \arrow[rd] &                    \\
\omega^{Log} \arrow[r]           & \omega^{Log + Mod}
\end{tikzcd}
\end{center}
where the map $\omega^{LP}$ is the morphism from the Levine-Pandharipande cobordism ring to the logarithmic cobordism ring, and the map $\omega^{Log}\rightarrow \omega^{Log +Mod}$ is the quotient map. Then the composition $\omega^{LP}\rightarrow \omega^{Log+Mod}$ is an isomorphism.
\end{theo}

\begin{proof}
First, note that $\omega^{\text{LP}}\rightarrow \omega^{\text{Log+Mod}}$ is an injection since the log chern numbers $c_\lambda(T^{\text{log}}_{(X,D)})$ are invariants of $\omega^{\text{Log+Mod}}$, and these restrict to the chern numbers $c_\lambda(T_X)$ on pairs with no boundary, which are precisely the invariants of $\omega^{\text{LP}}$. So it suffices to prove that the natural map $\omega^{\text{LP}}\rightarrow \omega^{\text{Log+Mod}}$ is a surjection. 

We proceed by induction on the dimension of the variety $\dim X = n$ in $(X,D)$. The case $\omega^{\text{LP}}_{1}\rightarrow \omega^{\text{Log+Mod}}_1$ follows from the fact $(\mathbb{P}^1,\emptyset) = 2(\mathbb{P}^1,\text{pt})$, and this can be seen by performing deformation to the normal cone for $\text{pt}\in\mathbb{P}^1$.

Next, suppose we've proven that for any $m<n$ the map  $\omega^{\text{LP}}_m\rightarrow \omega^{\text{Log+Mod}}_m$ is a surjection. Using our discussion preceeding Theorem 5.2, we know that $\omega^{\text{Log}}_n$ and therefore $\omega^{\text{Log+Mod}}_n$ is generated by the snc pairs associated to varieties with line bundles 
\begin{equation*}
    (Z,L_1,\dots,L_k)
\end{equation*}
where $\dim Z +k = n$. By \cite{LeeP}, the algebraic cobordism ring $\omega^{\text{LP}}_{k,1^k}$ of varieties with line bundles $(Z,L_1,\dots,L_k)$ is generated by elements obtained via taking a partition $\lambda\vdash \dim Z$ and taking a sub-partition $\mu$ of $\lambda$, so that the correspoinding basis element $\omega^{\text{LP}}_{k,1^k}$ is of the form 
\begin{equation*}
    (\mathbb{P}^{\lambda_1}\times \dots \mathbb{P}^{\lambda_{l(\lambda)}}, M_1,...,M_k)
\end{equation*}
such that: each $M_i$ is pulled back from one of the factors in the product $\mathbb{P}^{\lambda_1}\times \dots \mathbb{P}^{\lambda_{l(\lambda)}}$, each $M_i$ is pullbed back from a distinct factor, and each $M_i$ is either $\mathcal{O}(1)$ or $\mathcal{O}$. The snc pair corresponding to a basis element of $\omega^{\text{LP}}_{k,1^k}$ will, by applying Proposition 6.1, be a mix of product of elements of the form $(\mathbb{P}^l,\emptyset),(\mathbb{P}^l,\mathcal{O}(1)),(\mathbb{P}^l,\mathcal{O})$. We now use that deforming to the normal cone for a hyperplane $\mathbb{P}^{l}\subset \mathbb{P}^{l+1}$ gives us the relation
\begin{equation*}
    (\mathbb{P}^{l+1},\emptyset) = (\mathbb{P}^{l+1},\mathbb{P}^{l}) + (\mathbb{P}^{l},\mathcal{O}(1))
\end{equation*}
Assuming that $l<n-1$, our inductive hypothesis tells us that $ (\mathbb{P}^{l+1},\mathbb{P}^{l})$ can be written as a sum of terms with no boundary, and therefore the same is true for $(\mathbb{P}^{l},\mathcal{O}(1))$. 

The conclusion of the above argument is that to show $\omega^{\text{LP}}_n\rightarrow \omega^{\text{Log+Mod}}_n$ is surjective, it suffices to prove that $(\mathbb{P}^{n-1},\mathcal{O}(1))$ can be written as a sum of terms with no boundary. Equivalently, it suffices to show that $(\mathbb{P}^n,\mathbb{P}^{n-1}) = (\mathbb{P}^n,\emptyset) - (\mathbb{P}^{n-1},\mathcal{O}(1))$ can be written as a sum of terms with no boundary. This is where we use the relation $(\mathbb{P}^n, H_0\cup \dots \cup H_n)=0$. Performing deformation to the normal cone for $H_1\subset \mathbb{P}^n$ gives us the following relation:
\begin{equation*}
    (\mathbb{P}^n,H_0) = (\mathbb{P}^n,H_0\cup H_1) + (H_1,H'_0,\mathcal{O}(1))
\end{equation*}
where $H'_0 = H_1\cap H_0 = H_1$ is the induced boundary on $H_1$ and the snc pair associated to $(H_1,H'_0,\mathcal{O}(1))$ is given by taking $\mathbb{P}(\mathcal{O}_{H_1}(1)\oplus\mathcal{O})$ and letting the boundary be $\mathbb{P}(\mathcal{O}(1))\cup \eta^{-1}(H'_0)$ where $\eta: \mathbb{P}(\mathcal{O}_{H_1}(1)\oplus\mathcal{O}) \rightarrow H_1$ is the bundle projection. 

Now, if we perform deformation to the normal cone for $H'_0\subset H_1,$ we get that 
\begin{equation*}
    (H_1,\emptyset,\mathcal{O}(1)) = (H_1,H'_0,\mathcal{O}(1)) + (H'_0,\mathcal{O}(1),\mathcal{O}(1))
\end{equation*}
Recall also the identity 
\begin{equation*}
    (\mathbb{P}^n,\emptyset) = (\mathbb{P}^n,\mathbb{P}^{n-1}) + (\mathbb{P}^{n-1},\mathcal{O}(1))
\end{equation*}
Applying this to $(H_1,\emptyset,\mathcal{O}(1))$, we get that 
\begin{equation*}
    (H_1,H'_0,\mathcal{O}(1)) = (\mathbb{P}^n,\emptyset) - (\mathbb{P}^n,\mathbb{P}^{n-1}=H_0) - (H'_0,\mathcal{O}(1),\mathcal{O}(1))
\end{equation*}
Using the above expression we can now write:
\begin{align*}
    (\mathbb{P}^n,H_0) &= (\mathbb{P}^n,H_0\cup H_1) + (H_1,H'_0,\mathcal{O}(1))\\
    &= (\mathbb{P}^n,H_0\cup H_1) + (\mathbb{P}^n,\emptyset) - (\mathbb{P}^n,H_0) - (H'_0,\mathcal{O}(1),\mathcal{O}(1))
\end{align*}
Which can be re-written as 
\begin{equation*}
     (\mathbb{P}^n,H_0) = \frac{1}{2}((\mathbb{P}^n,H_0\cup H_1) + (\mathbb{P}^n,\emptyset) - (H'_0,\mathcal{O}(1),\mathcal{O}(1)))
\end{equation*}
Since $\dim H'_0 = n-2<n-1$, our inductive hypothesis tells us that $(H'_0,\mathcal{O}(1),\mathcal{O}(1))$ can be written as a sum of terms with no boundary. Therefore, it suffices to show that $(\mathbb{P}^n,H_0\cup H_1)$ can be written as a sum of terms with no boundary. Performing deformation to the normal cone for $H_2\subset \mathbb{P}^n$, we get the equality
\begin{equation*}
    (\mathbb{P}^n,H_0\cup H_1) = (\mathbb{P}^n,H_1\cup H_1\cup H_2) + (H_2,H'_0\cup H'_1,\mathcal{O}(1))
\end{equation*}
The term $(H_2,H'_0\cup H'_1,\mathcal{O}(1))$ can be handled just as how the term $(H_1,H'_0,\mathcal{O}(1))$ was handled. Continuing on in this fashion, this process eventually terminates since $(\mathbb{P}^n,H_0\cup\dots \cup H_n) = 0$ in $\omega^{\text{Log+Mod}}$, and each of the terms that we pick up are of the form 
\begin{equation*}
    (H_k,H'_0\cup\dots\cup H'_{k-1},\mathcal{O}(1))
\end{equation*}
and each such term is handled by using our inductive hypothesis and using our analysis of the term $(H_1,H'_0,\mathcal{O}(1))$. The conclusion is that we have an expression of the form 
\begin{equation*}
    (\mathbb{P}^n,H_0) = \text{rational number}(\sum\text{terms with no boundary})
\end{equation*}
and hence we are done.

\end{proof}

The main point of the above proof is that showing surjectivity of $\omega^{LP}_n\rightarrow \omega^{Log+Mod}_n$ is equivalent to showing that snc pair $(\mathbb{P}^n,\mathbb{P}^{n-1})$ lies in the image, i.e. can be written as a sum of terms with no boundary. Our proof then gives an algorithm for writing $(\mathbb{P}^n,\mathbb{P}^{n-1})$ as a sum of terms with no boundary. 
\begin{eg}
We go explicitly through our algorithm of writting $(\mathbb{P}^n,\mathbb{P}^{n-1})$ as a sum of terms with no boundary in the case when $n=2$. Let $L_0,L_1,L_2$ be the toric boundary of $\mathbb{P}^2$, then by performing deformation to the normal cone for $L_1$ we get the relation 
\begin{equation*}
    (\mathbb{P}^2,L_0) = (\mathbb{P}^2,L_0\cup L_1) + (L_1,p,\mathcal{O}(1)) =  (\mathbb{P}^2,L_0\cup L_1) + (\mathbb{P}^1,p,\mathcal{O}(1)) 
\end{equation*}
where $p = L_0\cap L_1$. By performing deformation to the normal cone for $p\in L_1 = \mathbb{P}^1$, we get the relation
\begin{equation*}
    (\mathbb{P}^1,\emptyset,\mathcal{O}(1)) = (\mathbb{P}^1,p,\mathcal{O}(1)) + (p,\emptyset,\mathcal{O}(1)|_p,\mathcal{O}(1)|_p)
\end{equation*}
The term $(p,\emptyset,\mathcal{O}(1)|_p,\mathcal{O}(1)|_p)$ corresponds to the snc surface pair given by 
\begin{equation*}
(\mathbb{P}^1,0)\cdot (\mathbb{P}^1,0) = \frac{1}{4}(\mathbb{P}^1,\emptyset)^2 = \frac{1}{4}(\mathbb{P}^1\times\mathbb{P}^1,\emptyset)
\end{equation*}
where here we used the relation $(\mathbb{P}^1,p) = \frac{1}{2}(\mathbb{P}^1,\emptyset)$. So we get that 
\begin{align*}
(\mathbb{P}^1,p,\mathcal{O}(1)) &= (\mathbb{P}^1,\mathcal{O}(1)) - \frac{1}{4}(\mathbb{P}^1\times\mathbb{P}^1,\emptyset) \\
&= (\mathbb{P}^2,\emptyset)-(\mathbb{P}^2,\mathbb{P}^1) - \frac{1}{4}(\mathbb{P}^1\times\mathbb{P}^1,\emptyset)
\end{align*}
Plugging this into the expression 
\begin{equation*}
    (\mathbb{P}^2,L_0) = (\mathbb{P}^2,L_0\cup L_1) + (L_1,p,\mathcal{O}(1)) =  (\mathbb{P}^2,L_0\cup L_1) + (\mathbb{P}^1,p,\mathcal{O}(1)) 
\end{equation*}
we get that 
\begin{equation*}
    2(\mathbb{P}^2,\mathbb{P}^1) = (\mathbb{P}^2,\emptyset) - \frac{1}{4}(\mathbb{P}^1\times\mathbb{P}^1,\emptyset) + (\mathbb{P}^2,L_0\cup L_1)
\end{equation*}
Performing deformation to the normal cone for $L_2$ and using $(\mathbb{P}^2,L_0\cup L_1\cup L_2) = 0$ we get 
\begin{equation*}
    (\mathbb{P}^2,L_0\cup L_1) = 0 + (\mathbb{P}^1,p+q,\mathcal{O}(1))
\end{equation*}
where $L_2 = \mathbb{P}^1$ and $p+q = (L_0\cup L_1)\cap L_2$. We can write 
\begin{align*}
(\mathbb{P}^1,p+q,\mathcal{O}(1)) &= (\mathbb{P}^1,p,\mathcal{O}(1)) - (q,\emptyset,\mathcal{O}(1)|_q,\mathcal{O}(1)|_q) \\
&=(\mathbb{P}^1,p,\mathcal{O}(1)) - \frac{1}{4}(\mathbb{P}^1\times\mathbb{P}^1,\emptyset)
\end{align*}
So we can now write 
\begin{align*}
    2(\mathbb{P}^2,\mathbb{P}^1) = (\mathbb{P}^2,\emptyset) -\frac{1}{4}(\mathbb{P}^1\times\mathbb{P}^1,\emptyset) +(\mathbb{P}^1,p+q,\mathcal{O}(1))
\end{align*}
And plugging in our expression for $(\mathbb{P}^1,p+q,\mathcal{O}(1))$ we end up at 
\begin{equation*}
    3(\mathbb{P}^2,\mathbb{P}^1) = 2(\mathbb{P}^2,\emptyset) - \frac{3}{4}(\mathbb{P}^1\times\mathbb{P}^1,\emptyset)\Rightarrow (\mathbb{P}^2,\mathbb{P}^1) = \frac{2}{3}(\mathbb{P}^2,\emptyset) - \frac{1}{4}(\mathbb{P}^1\times\mathbb{P}^1,\emptyset)
\end{equation*}
We can check that our expression for $(\mathbb{P}^2,\mathbb{P}^1)$ as a sum of terms with no boundary makes sense by using that the map
\begin{equation*}
\omega^{\text{Log+Mod}}_2\rightarrow \mathbb{Q}\times \mathbb{Q},\quad (S,D)\mapsto (c_1(T^{\text{Log}})^2,c_2(T^{\text{log}}))
\end{equation*}
is an isomorphism and checking that our expression above agrees on the level of chern invariants:
\begin{equation*}
    (4,1) = \frac{2}{3}(9,3) - \frac{1}{4}(8,4)
\end{equation*}

\end{eg}

\section{The log cobordism/modification ring and DT invariants}
The fact that the map $\omega^{\text{LP}}_3 \rightarrow \omega^{\text{Log+Mod}}_3$ is an isomorphism tells us that $\omega^{\text{Log+Mod}}_3$ is generated by elements with no boundary. This allows us to give another simpler proof of the fact that the generating function for zero-dimensional logarithmic DT invariants is given by the simple identity:
\begin{equation*}
    Z(X,D) = M(q)^{\int_X c_3(T^\text{log}_X\otimes K^\text{log}_X)}
\end{equation*}
Similar to the logarithmic birational invariance of Gromov-Witten invariants \cite{AW18}, if $(X',D')\rightarrow (X,D)$ is a logarithmic modification, then $\text{Z}(X',D') = \text{Z}(X,D)$. Thus the invariant $\text{Z}(X,D)$ factors through the quotient $\omega^{\text{log}}_3\rightarrow \omega^{\text{Log+Mod}}_3$, which has a much simpler structure.
\begin{theo}
Let $X$ be a smooth projective threefold and $D\subset X$ an snc divisor. Consider the generating function
\begin{align*}
    \text{Z}(X,D) = \sum_n \deg [\text{Hilb}^n(X,D)]^{\text{vir}}q^n
\end{align*}
Then we have 
\begin{equation*}
    Z(X,D) = M(q)^{\int_X c_3(T^\text{log}_X\otimes K^\text{log}_X)}
\end{equation*}
\end{theo}

\begin{proof}
The invariant $Z(X,D)$ gives us a homomorphism $\omega^{\text{Log+Mod}}_3\rightarrow \mathbb{Q}[[q]]$. The formula is known to hold for threefolds with no boundary $(X,\emptyset)$. However, $\omega^{\text{Log+Mod}}_3$ is generated by threefolds with no boundary, hence the formula is true for all elements.
\end{proof}

\begin{rem}
It is interesting to note that we did not invoke that the formula is known for pairs $(X,D)$ where $D$ is smooth (has one irreducible component). Thus, if one wanted to prove to the formula for the case when $D$ is smooth, it is (counter-intuitively) simpler to just pass to the full snc boundary case and then use all of the constraints imposed from log geometry (degeneration and modification invariance). 
\end{rem}

\end{document}